\documentclass[12pt]{article}

\usepackage{amsmath, amsthm,amsbsy,amsfonts,graphics,epsfig,float,url,color}
\usepackage{graphicx}
\usepackage{subcaption}
\usepackage{enumitem}
\usepackage{authblk}
\usepackage{rotating}
\usepackage[margin=1in]{geometry}
\allowdisplaybreaks

\newtheorem{theorem}{Theorem}[section]
\newtheorem{proposition}[theorem]{Proposition}
\newtheorem{remark}[theorem]{Remark}

\newcommand{\rod}{\text{RoD}}

\title{ An early warning system for multivariate time series with sparse and non-uniform sampling}
\author{Andrew Roberts}
\author{Sasanka Are}
\affil{Cerner Corporation}

\begin{document}
\maketitle

\begin{abstract} 
In this paper we propose a new early warning test statistic, the ratio of deviations (RoD), which is defined to be the root mean squared of successive differences divided by the standard deviation.  We show that RoD and autocorrelation are asymptotically related, and this relationship motivates the use of RoD to predict Hopf bifurcations in multivariate systems before they occur.  We validate the use of RoD on synthetic data in the novel situation where the data is sparse and non-uniformly sampled.  Additionally, we adapt the method to be used on high-frequency time series by sampling, and demonstrate the proficiency of RoD as a classifier.  

\end{abstract}

\section{Introduction}
The prospect of identifying a tipping point in an observed system prior to its occurrence has recently excited the scientific community.  Deterioration of a stable state may indicate the proximity of a tipping point, although neither necessarily implies the other \cite{boettiger2013}.  The recent emphasis on early detection of tipping points has its roots in applications that are sensitive to anthropogenic activity, such as climate and ecology \cite{dakos2012,dakos2008}.  Applications where that human intervention could prevent an undesirable regime shift makes the idea of early warning particularly enticing, and these concepts are spreading to other applications where intervention is critical, e.g. medicine \cite{chen2012,van2014}. 

Popular early warning signals include observing an increasing trend in a test statistic such as autocorrelation, variance, or skewness \cite{boettiger2013,kefi2013,lenton2011,van2014}, with the theory based in dynamical systems.  As a dynamical system  approaches a tipping point, the equilibrium state loses stability, and the destabilization is reflected in the eigenvalues of the Jacobian. The time it takes for small perturbations to return to the equilibrium state increases as the leading eigenvalue approaches the boundary between stable and unstable in the complex plane---either the imaginary axis for systems modeled in continuous time or the unit circle for map dynamics.  This phenomenon is known as \emph{critical slowing down}, referring to the decreased rate of return.  The requisite details on the theory of dynamical systems can be found in most introductory textbooks on the subject such as those by Strogatz \cite{strogatz2014} or Glendinning \cite{glendinning1994}.

Developing early warning signals to detect catastrophic changes, with the most common being a fold or saddle-node bifurcation, has been the primary focus of the field to date.  Catastrophic changes are necessarily \emph{irreversible}, and thus are accompanied by a 'point of no return.'  Other bifurcations (e.g., supercritical Hopf bifurcations) are reversible, and thus are not always considered tipping points.  However, in many applications it is still desirable to identify preventable or predictable changes in stability, even if they are reversible.  

In this paper we present a new test statistic for use as an early warning signal, called the \emph{ratio of deviations} (RoD), relating the root mean squared of successive deviations to the standard deviation of a time series.  We show that for a time series generated from a weakly stationary lag-1 autoregressive process, RoD is asymptotically related to lag-1 autocorrelation.  According to this relationship, as autocorrelation decreases the RoD increases.  Inspired by the asymptotic relationship, we examined what an increase in the RoD indicates about the future values of a time series.  We propose that an increase in RoD is indicative of a change in the nature of oscillations.  Thus, the RoD should be able to detect passage through a Hopf bifurcation.  

Supercritical Hopf bifurcations are reversible, and as such have not been a primary focus in the early warning literature.  Subcritical Hopf bifurcations may be accompanied by a catastrophic saddle-node bifurcation of periodic orbits, with a region of bistability where the two attractors are a stable equilibrium point and a stable periodic orbit.  However, even in the supercritical case, there is a hysteresis effect with a dynamic Hopf bifurcation in that a trajectory can remain near the newly unstable equilibrium for quite some time \cite{baer1989}.  Figure \ref{fig:del_hopf} depicts the phenomenon. Thus, even if a Hopf bifurcation is not detected until after it has occurred, there may still be enough time left to prevent any negative effects of a regime change.  Conversely, `reversing' a Hopf bifurcation after it is visually apparent (i.e., after a trajectory is near the periodic orbit), could require a significant shift in the parameter.  

\begin{figure}
	\centering
	\includegraphics[width=\textwidth]{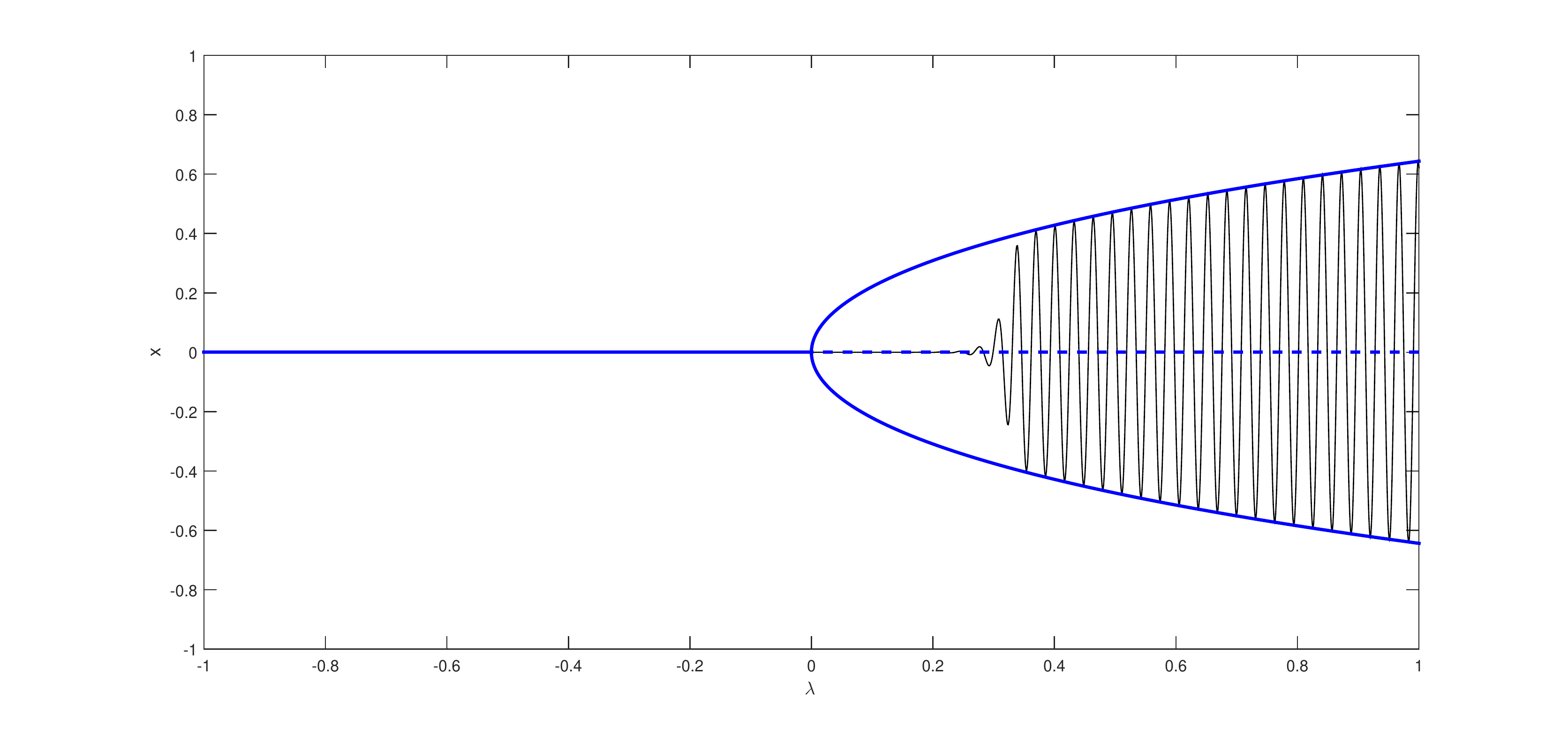}
	\caption{A dynamic Hopf bifurcation.  Blue lines and curves depict the bifurcation diagram of a system undergoing a Hopf bifurcation.  Solid blue lines/curves indicate stable sets for fixed $\lambda$, with the curve for $\lambda >0$ indicating the maximum and minimum $x$ values of an attracting periodic orbit.  The dashed blue line ($\lambda >0$) indicates the unstable equilibrium.  A trajectory (black) is depicted when $\lambda = \lambda(t)$.    }
	\label{fig:del_hopf}
\end{figure}

RoD is designed to be used on short time series with irregularly sampled observations, and we extend its use to longer, high-frequency time series through sampling.  To demonstrate the predictive skill of RoD, we simulated a system known to exhibit a supercritical Hopf bifurcation as a parameter passes through a critical value, plus control runs where the bifurcation parameter remains constant.  We ran two types of tests: (a) to evaluate use on short time series and (b) to evaluate the extension to longer time series.  RoD (paired with another test statistic) performed adequately as a predictor on the short time series.  Its performance on longer time series is exceptional.  


The paper will proceed as follows: Section \ref{sec:rod} describes the theory behind the ratio of deviations as an early warning signal.  Section \ref{sec:validation} discusses the performance on synthetic data.  
We conclude with a discussion in Section \ref{sec:discussion}.

\section{Ratio of Deviations}
\label{sec:rod}

Let ${X_{t_k}}$ be a univariate time series where $k = 1, \ldots, n, \ldots$.  The RoD relates the root mean squared of successive deviations (RMSSD), $\nu$, and standard deviation, $\sigma$.  By definition, we have
\begin{align}
	\nu^2(t_n) &= \frac{1}{n-1} \sum_{i=2}^n (X_{t_i} -X_{ t_{i - 1} } )^2, \text{ and} \\
	\sigma^2(t_n) &= \frac{1}{n} \sum_{i=1}^n ( X_{t_i} - \mu_n)^2,
\end{align}
where $\mu = \text{E}[X]$.  RoD is defined as
\begin{equation}
	\rod(t_n) = \frac{\nu(t_n)}{\sigma(t_n)}.
\end{equation}

\begin{proposition}
	\label{prop:rod}
If ${X_{t_i}}$ is a weakly stationary process, then 
\begin{equation}
	\lim_{n\rightarrow \infty} \rod^2(t_n) = 2 ( 1- \rho_X(1)), 
\end{equation}
where $\rho_X(1)$ is the lag-1 autocorrelation of $X_{t_i}$.  
\end{proposition}

\begin{proof}
Both standard deviation and RMSSD are independent of the mean, so we assume that $\text{E}[X]=0$ zero without loss of generality.  Since $X_{t_i}$ is weakly stationary, we have that
\begin{equation*}
\rho_X(1) = \frac{\text{E}[X_{t_i}X_{t_{i-1}}]}{\sigma^2}.
\end{equation*}
Thus,
\begin{align*}
	\lim_{n\rightarrow \infty} \rod^2(t_n) 
	&= \lim_{n\rightarrow \infty}  
		\frac{ \displaystyle \frac{1}{n-1} \sum_{i=2}^n (X_{t_i} - X_{t_{i - 1}})^2}
			{\displaystyle \frac{1}{n} \sum_{i=1}^n X_{t_i}^2} \\
	&= \lim_{n\rightarrow \infty} 
		\frac{ \displaystyle \frac{1}{n-1} \sum_{i=2}^n X_{t_i}^2 +\frac{1}{n-1} \sum_{i=1}^{n-1}
			X_{t_i}^2 - 2 \frac{1}{n-1} \sum_{i=2}^n X_{t_{i-1}} X_{t_i} }
			{\displaystyle \frac{1}{n} \sum_{i=1}^n X_{t_i}^2} \\
	&= \lim_{n\rightarrow \infty} 
		\left( 
		\frac{ \displaystyle \frac{1}{n-1} \sum_{i=2}^n X_{t_i}^2 }{\displaystyle \frac{1}{n} \sum_{i=1}^n X_{t_i}^2}
		 + \frac{ \displaystyle\frac{1}{n-1} \sum_{i=1}^{n-1} X_{t_i}^2}{\displaystyle \frac{1}{n} \sum_{i=1}^n X_{t_i}^2}
		  - 2 \frac{ \displaystyle \frac{1}{n-1} \sum_{i=2}^n X_{t_{i-1}} X_{t_i} }{\displaystyle \frac{1}{n} \sum_{i=1}^n X_{t_i}^2} \right) \\
	&= 2 \frac{\sigma^2}{\sigma^2} - 2 \rho_X(1) \\
	&= 2(1-\rho_X(1) ).
\end{align*}
\end{proof}

\begin{remark}
We can compute an AR(1) coefficient and RoD for any finite time series, however the preceding result relies on $\sigma^2$ (eventually) being time-independent.
\end{remark}

\subsection{Detecting Changes in an Underlying System}
By examining a ratio of two distinct measures of dispersion, RoD is designed to detect a change in the nature of deviations in a time series.  We say RoD detects a change in the variable $X_t$ at time $t_k$ both $\rod(t_k) > \rod(t_{k-1})$.  Note that we only require a single increase in RoD rather than an increasing trend, in contrast with the standard practice for other early warning signals such as autocorrelation.  

It is possible that RoD increases even if both standard deviation and RMSSD decrease.  To remedy this, we recommend pairing RoD with a tandem value, similar to using autocorrelation in conjunction with standard deviation \cite{ditlevsen2010}.  We have explored requiring an increase in standard deviation or an increase in RMSSD.  Requiring an increase in standard deviation reduces the number of positives more than requiring RMSSD to increase does.  A third option is to set restrictions on the range of values a given variable can take, and require a value not meeting those restrictions.  In the third case, the reduction in positives will clearly be dependent upon the range.  The tandem metric that makes the most sense may depend on the application, a priori knowledge of the underlying system, or properties of the observations (especially on short time series).  

For a univariate time series $X_t$, we assume the underlying model is of the form
\begin{equation}
X_{t_k} = a_{t_k} X_{t_{k-1}} + \xi_{t_k}.
\end{equation}  
A change in $a_t$ will affect the stability of the system, possibly causing an increase in RoD depending on the magnitude and direction of the change.  Alternatively, RoD could detect a change due to a rare event where $\xi_{t_k}$ takes a value far from its mean.  

Because detection can be triggered due to a random event rather than deterioration of a stable state, RoD is best applied to multivariate systems.  Consider the multivariate linear system
\begin{equation}
\vec{X}(t_k) = A(t_k) \vec{X}(t_{k-1}) + \vec{\xi}(t_k),
\end{equation}
where $\vec{X}(t) = (X_1(t),\ldots,X_n(t) ) \in \mathbb{R}^n$, $\vec{\xi}(t) = (\xi_1(t),\ldots,\xi_n(t) ) \in \mathbb{R}^n$, and $A(t) = \left( a_{ij}(t) \right).$
By calculating the RoD for each variable $X_i(t)$ individually, we are able to detect changes in particular subsystems.  If we only detect a change in one variable, we conclude that it is likely due to the noise term.  However, if we detect a change in all variables (or all variables of a subsystem), it likely indicates a structural change in the system.  Furthermore, because we will use RoD on short time series with long times between observations, we would expect the effects of noise to subside before the next observation if the system were stable.

Note that we are detecting changes by observing an increase in the RoD of each univariate time series.  Proposition \ref{prop:rod} suggests that an increase in the RoD roughly corresponds with a decrease in autocorrelation.  If the underlying system loses stability due to decreasing autocorrelation, we should expect the trajectory to develop oscillations of increasing amplitude.  Thus we anticipate that RoD is a mechanism for detecting Hopf bifurcations.

\section{Validation on Synthetic Data}
\label{sec:validation}
We evaluated RoD as a test statistic on a suite of systems that are known to exhibit Hopf bifurcations as a parameter $\lambda$ passes through a critical value $\lambda_c$, including an additive white noise term in each equation.  All systems were simulated according to the Euler-Maruyama method using the `deSolve' package in R.

\subsection{A Simple Example: Hopf Normal Form}
\label{sec:simple}
First we demonstrate the method on a simple example: the Hopf normal form with additive white noise.  The equations of the Hopf normal form are   
\begin{equation}
	\label{normal2}
	\begin{array}{rl}
		dx &= [\lambda(t) x - y + 2 \eta x (x^2+y^2) - x (x^2+y^2)^2] dt + \eta dW_1 \\
		dy &= [x + \lambda(t) y + 2 \eta y (x^2+y^2) - y (x^2+y^2)^2] dt +  \eta dW_2,
	\end{array}
\end{equation}
and we take $\eta =  0.25.$  The system was simulated from $t=0$ to $t=100$ using time steps of 0.05.  We ramped $\lambda(t)$ linearly from -1 at $t=0$ to 1 at $t=100$, so the system undergoes a dynamic Hopf bifurcation at $(t,\lambda)=(50,0)$.  The time series $x(t)$ is plotted in Figure \ref{fig:example}.  We then sampled the time series for both $x$ and $y$ at times $t_i$, where $t_0 = 0$ and $t_{i+1} = t_i + \text{unif}(4,8)$ to generate sample time series $(x_i,y_i)$.  The $x_i$ are depicted by triangles in Figure \ref{fig:example}.  Let $X_n = \{ x_0, x_1, \ldots, x_n \}$ and define $Y_n$ similarly.  We computed the RoD for $X_n$ and $Y_n$ for each $n$.  We say we detect a change in the variable $x$ if both $\text{RoD}(X_{n+1}) > \text{RoD}(X_n)$ and $\sigma(X_{n+1})>\sigma(X_n)$ (similarly for $y$).  The red triangle in Figure \ref{fig:example} indicates the observation when a change was detected in both variables.   

\begin{figure}
	\centering
	\includegraphics[width=\textwidth]{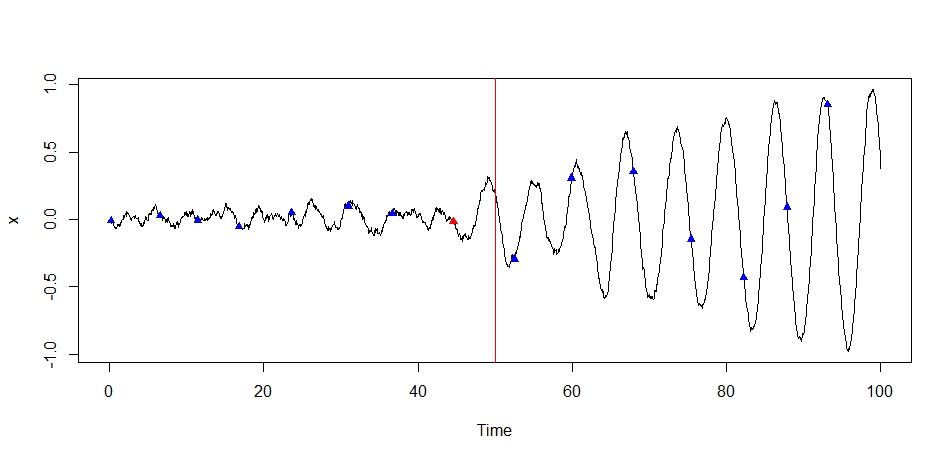}
	\caption{A simple example of RoD.  The black curve is $x(t)$ generated from \eqref{normal2} with $\sigma = 0.25$ and $\lambda(t) = t/50 - 1.$  The red vertical line at $t=50$ indicates the time when the system undergoes a dynamic Hopf bifurcation.  The triangles indicate a sample time series where the system is observed randomly between 4 and 8 time units after the previous observations.  The red triangle indicates the observation when RoD and standard deviation increase for both $x$ and $y$.  }
	\label{fig:example}
\end{figure}

\subsection{Experiments on a 3D Variation of the Van der Pol System} 
\label{sec:sys3}
A Hopf bifurcation can be observed in a 2-dimensional subsystem, at least locally (possibly after a change of coordinates).  However, in practice it is not necessarily clear which 2-dimensional subsystem one should consider nor what coordinate transformation should be used.  We tested the performance of RoD as a warning signal on two different parametrizations of 3D variation of the Van der Pol system.  

In certain parameter regimes, the Van der Pol system is an excitable system, meaning that a small perturbation (in a particular direction) can lead to a big oscillation \cite{glendinning1994,strogatz2014}.  Excitability is related to a separation of time scales and phenomenon known as \emph{canard explosion}.  For our purposes, it is sufficient to think of a canard explosion as Hopf bifurcation in which the amplitude of the periodic orbits grows exponentially in terms of distance in parameter space from the bifurcation.  A reader interested in delving deeper into the rich literature of canard-related phenomena can begin with works by Krupa and Szmolyan (e.g., \cite{krupa2001}) or those of Wechselberger (e.g., \cite{wechselberger2007}).  

The variant of the Van der Pol system we examine is
\begin{equation}
\label{vdp3}
\begin{array}{rl}
dx &= \frac{1}{a} (3x-x^3-y) dt + \sigma dW_1 \\
dy &= [x - \lambda(t)] dt +  \sigma dW_2 \\
dz &= (x-z)dt + \sigma dW_3,
\end{array}
\end{equation}
where $a$ is a time-scale parameter that affects the growth rate of the amplitude of periodic orbits.  We will refer to system \eqref{vdp3} with $a=10$ as `normal' because the amplitudes grow as expected according to a normal Hopf bifurcation.  We will refer to the system with $a=1$ as `excitable' or `with canard explosion' because the amplitudes grow much faster than normal.  Figure \ref{fig:can_exp} shows how the amplitudes of the periodic orbits grow after bifurcations in each case.  

\begin{figure}[t]
	\centering
	\begin{subfigure}[b]{0.45\textwidth}
		\includegraphics[width=\textwidth,height=2in]{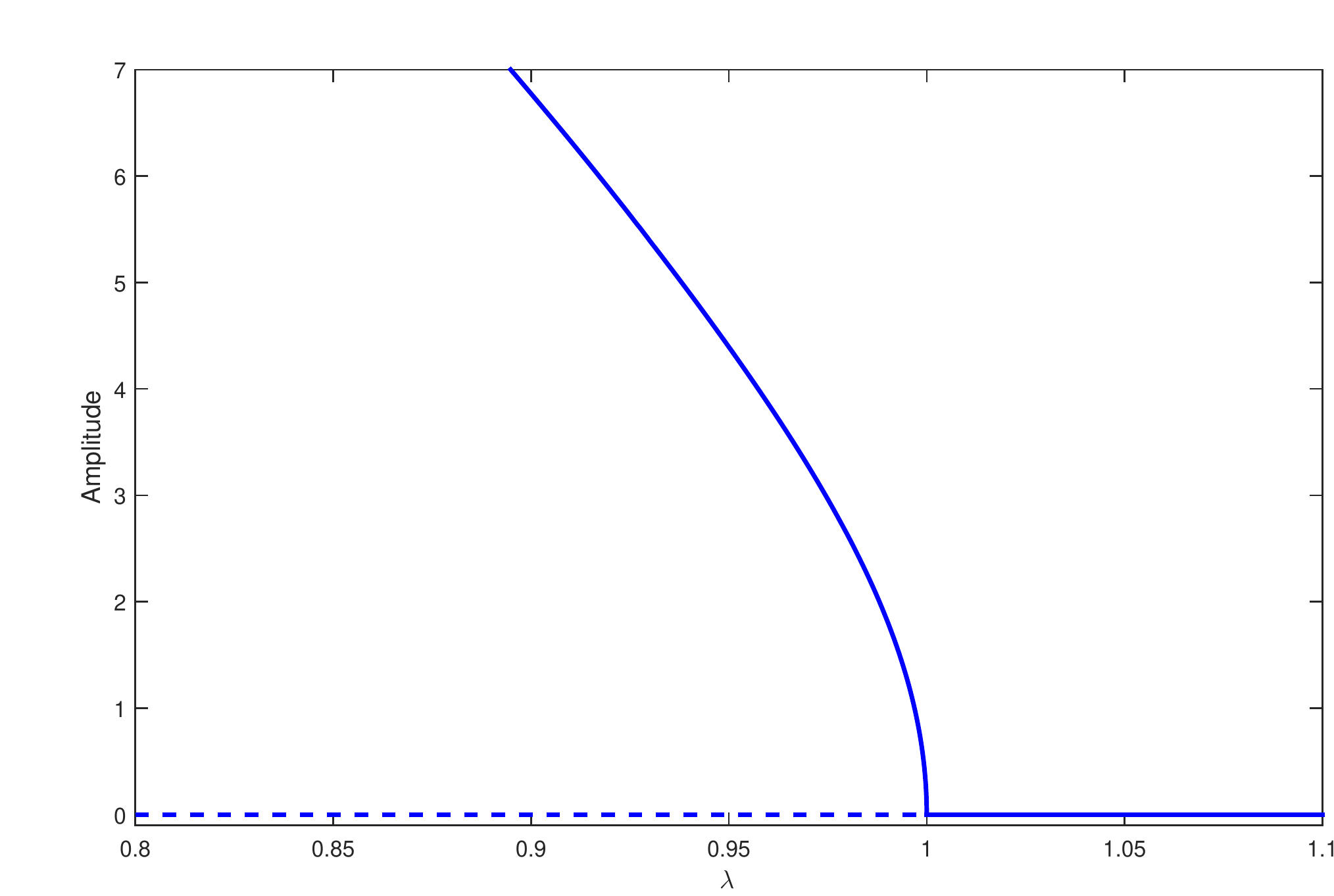}
		\caption{$a=10$ (Normal).}
		\label{fig:normal}
	\end{subfigure}
	~ 
	\begin{subfigure}[b]{0.45\textwidth}
		\includegraphics[width=\textwidth,height=2in]{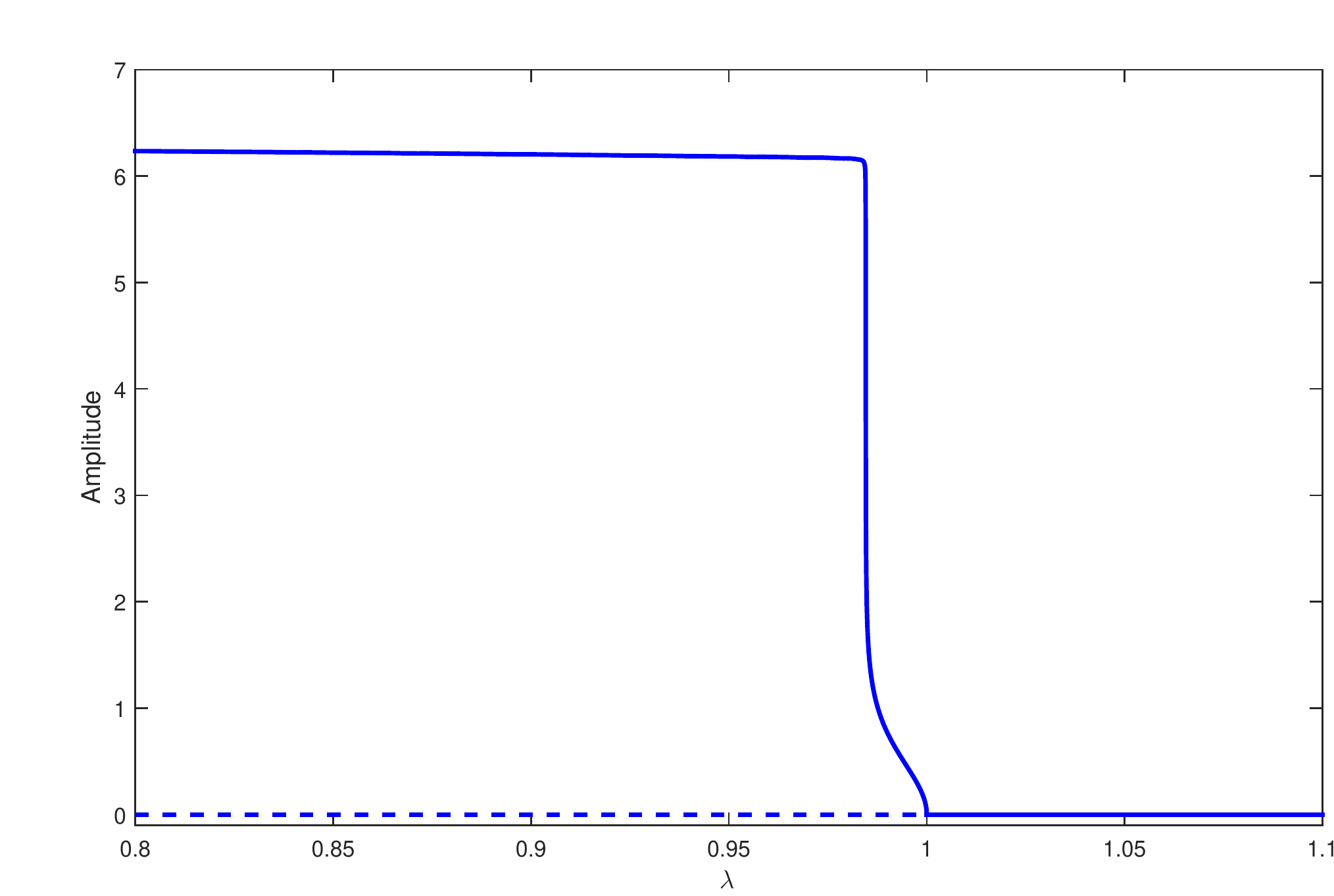}
		\caption{$a=1$ (Excitable).}
		\label{fig:excitable}
	\end{subfigure}
	\caption{Growth of periodic orbits after Hopf bifurcation at $\lambda = 1$ in \eqref{vdp3}.}   
	\label{fig:can_exp} 
\end{figure}

We ran 200 simulations of each parametrization for each of five prescribed noise intensities ($\sigma = 0,0.01,0.05,0.1,0.25$.) In half of the runs $\lambda$ was ramped to pass through a Hopf bifurcation, and the other half were control runs where $\lambda$ was held constant (i.e., $\lambda(t) \equiv \lambda_0 $).  The initial value $\lambda(0) = \lambda_0 = 1.2$ was chosen so that the system had an attracting equilibrium point, and we used the equilibrium point as the initial condition for the model.  Each simulation ran for 2000 dimensionless time units with steps of 0.05 units.  For the simulations when $\lambda$ was ramped, we chose $d\lambda/dt$ so that $\lambda(1000) = \lambda_c = 1$, i.e., $\lambda$ passed through the critical value exactly half-way through the simulation.

As stated earlier, we do not want to compute the RoD on the whole trajectory with high-frequency observations.  Instead, we observe the trajectories iteratively at random times, so that $t_{i+1} = t_i + \Delta T$ where $t_0 = 0$, $\Delta T = unif(\alpha, \beta),$ and $\beta > \alpha > 0.$  For each system we ran experiments varying $(\alpha,\beta)$, generating 100 time series from each trajectory for each of the values in Table \ref{tbl:obs}.  Additionally, we experimented with the window over which we used observations to calculate the RoD, using windows of 250, 500, 750, and 1000 time units.  
  
\begin{table}[b] 
	\begin{center}
		\addtolength{\tabcolsep}{1mm}
		\renewcommand{\arraystretch}{1.2}
		\begin{tabular}{|c|c|c|c|}
			\hline
			\textbf{Experiment}     & \textbf{$\alpha$}      	& \textbf{$\beta$} & \textbf{Mean Sampling Time = $(\alpha+\beta)/2$} 			\\
			\hline
			\hline  
			1 	 					& 20  						& 40               & 30  		\\
			\hline
			2 	 					& 25                 		& 50               & 37.5   		\\
			\hline
			3  						& 25  						& 75               & 50 		\\
			\hline
			4  						& 50  						& 100              & 75			\\
			\hline
		\end{tabular}
		\caption{Values used for experiments generating observations at random times.} 
		\label{tbl:obs}
	\end{center}
\end{table}

\begin{sidewaysfigure}
	\includegraphics[width=\columnwidth]{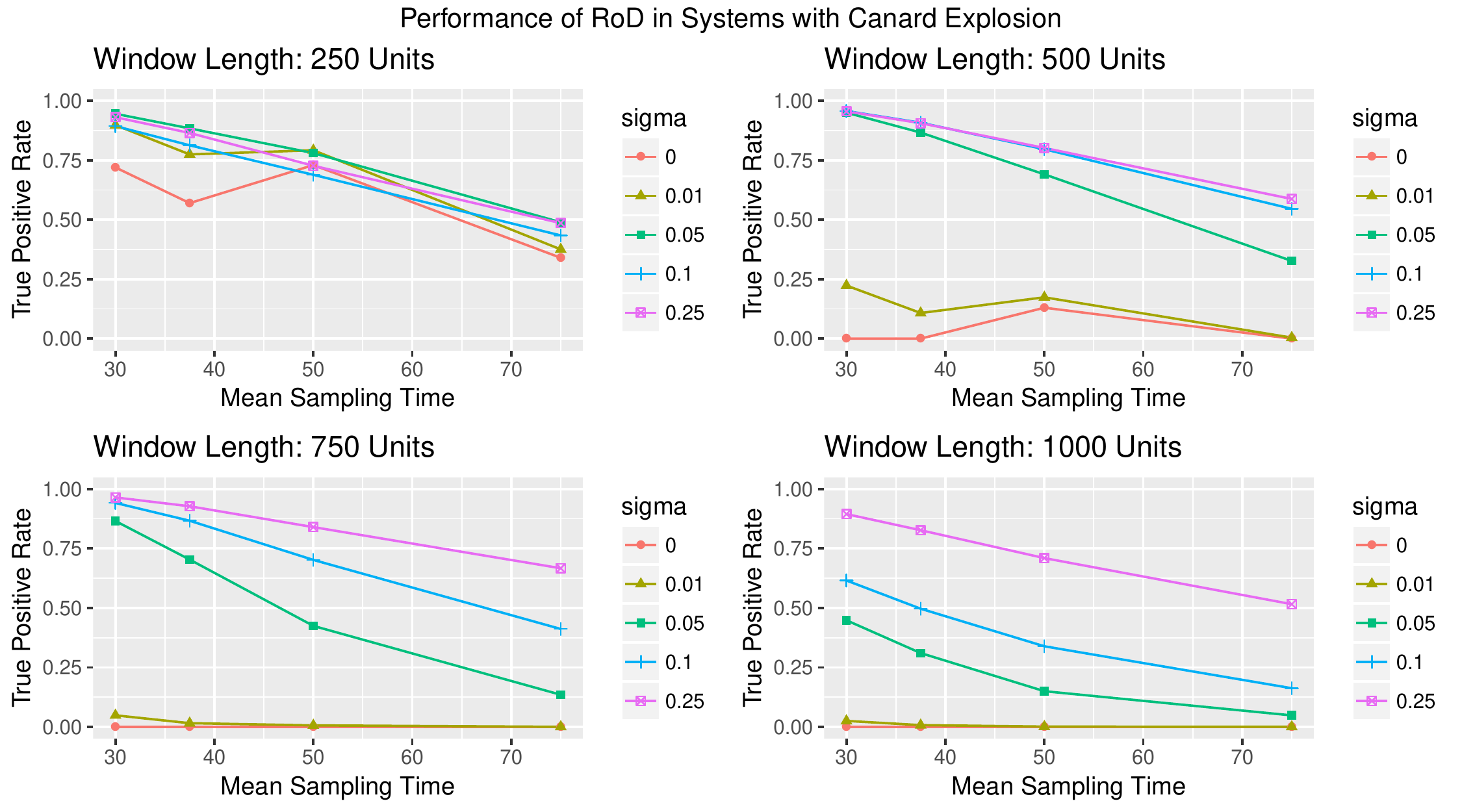}
	\caption{Rate of true positives on sampled time series from system \eqref{vdp3} with $a=1$.}
	\label{fig:exp_tp}
\end{sidewaysfigure}

\begin{sidewaysfigure}
	\includegraphics[width=\columnwidth]{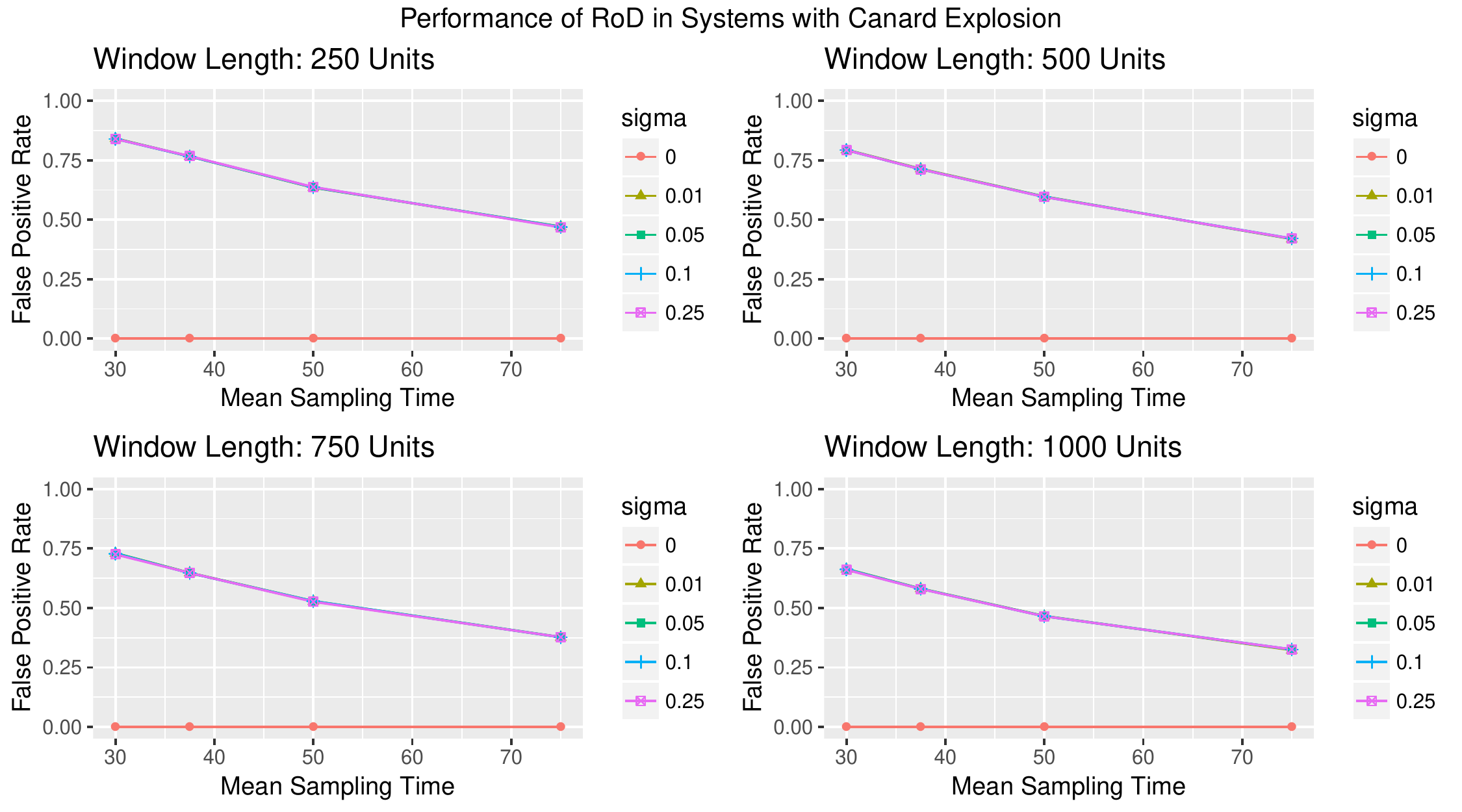}
	\caption{Rate of false positives on sampled time series from system \eqref{vdp3} with $a=1$.}
	\label{fig:exp_fp}
\end{sidewaysfigure}

\begin{sidewaysfigure}
	\includegraphics[width=\columnwidth]{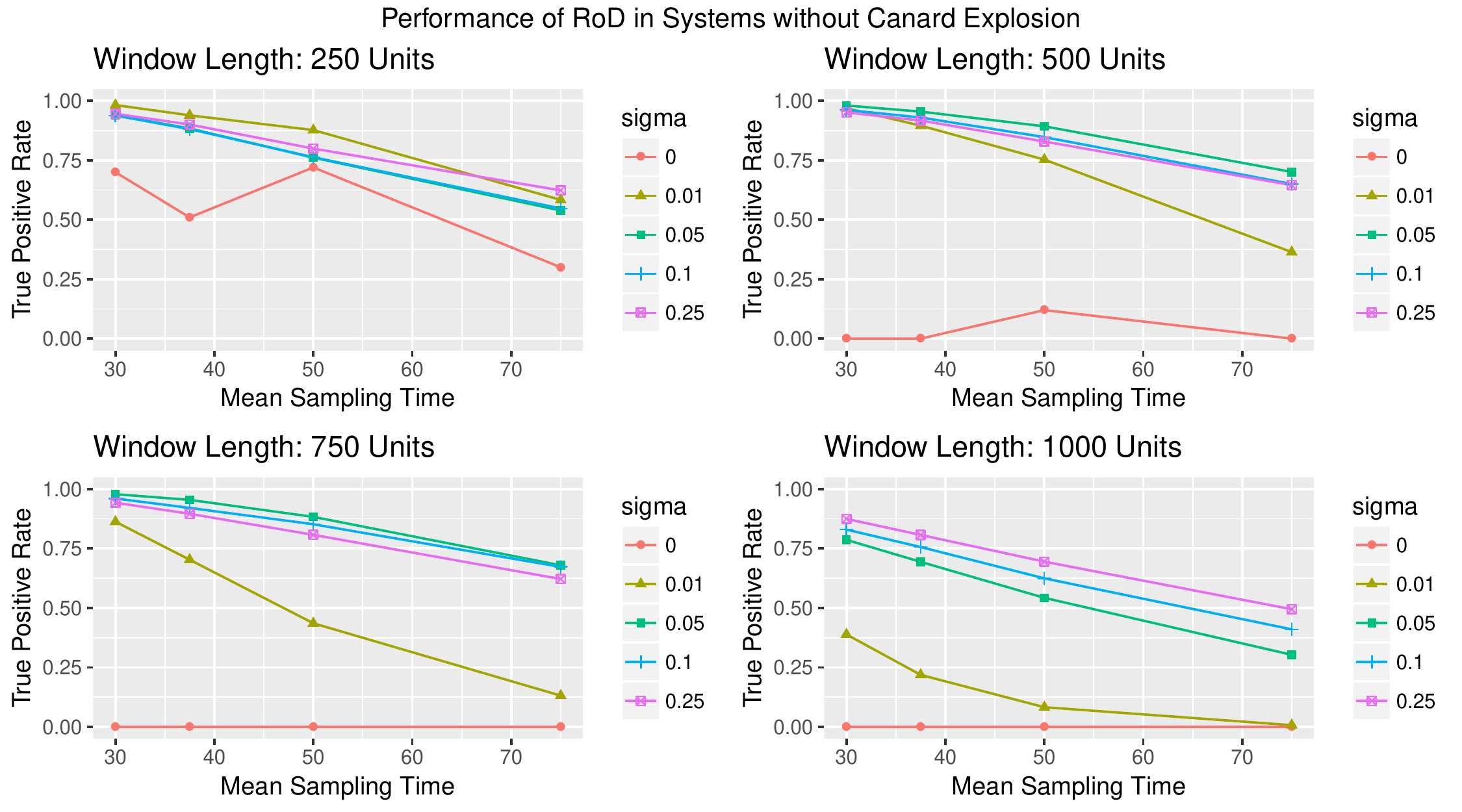}
	\caption{Rate of true positives on sampled time series from system \eqref{vdp3} with $a=10$.}
	\label{fig:std_tp}
\end{sidewaysfigure}

\begin{sidewaysfigure}
	\includegraphics[width=\columnwidth]{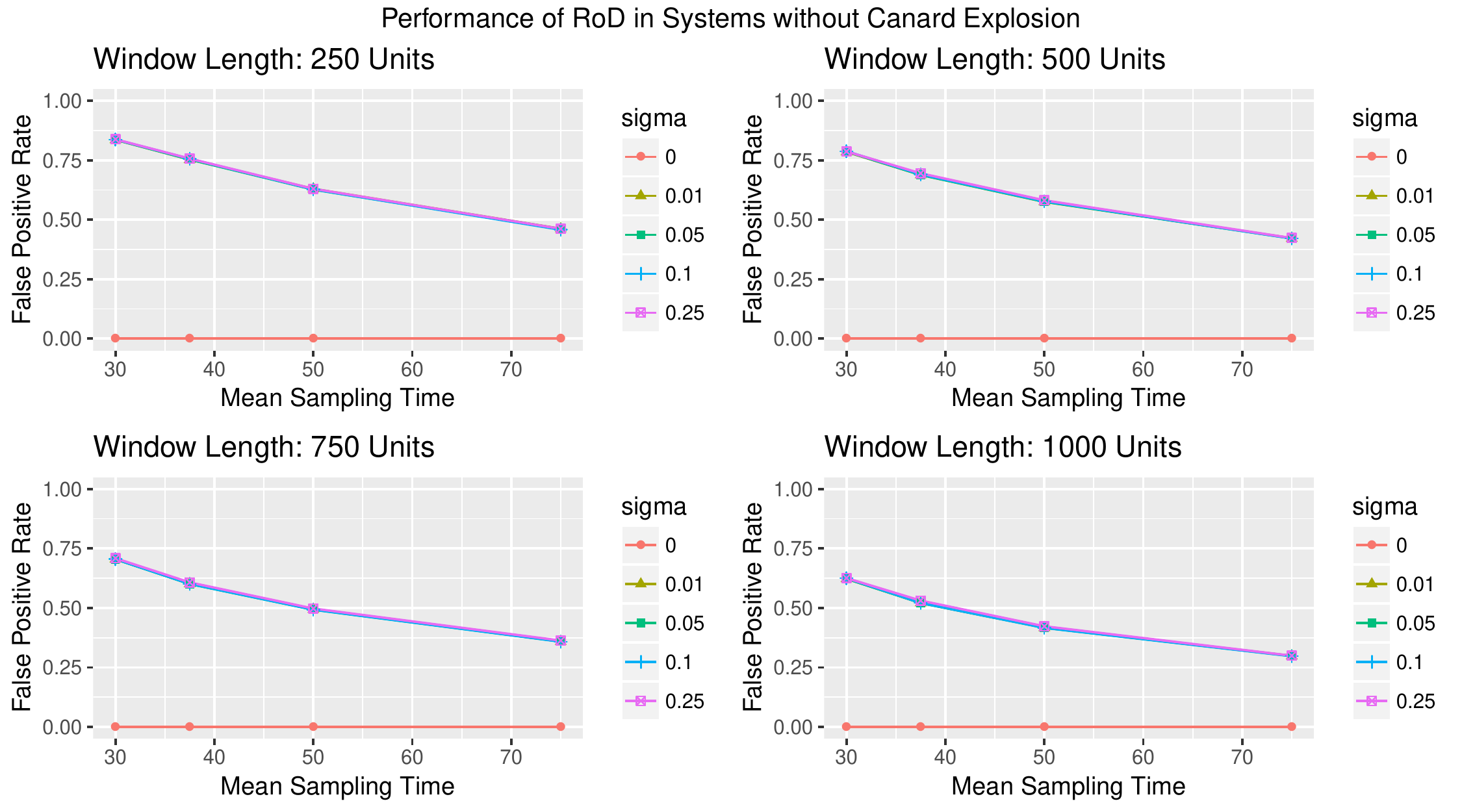}
	\caption{Rate of false positives on sampled time series from system \eqref{vdp3} with $a=10$.}
	\label{fig:std_fp}
\end{sidewaysfigure}

Since we want to detect bifurcations before they occur at $t=1000$, the window of 1000 time units represents an experiment using the entire trajectory.  Figures \ref{fig:exp_tp}-\ref{fig:std_fp} break down the results of these experiments by parametrization, window length, and observation rate when using RoD with RMSSD. 

\subsection{Experiments with High-frequency Data}
Next we turn our attention to experiments with high-frequency data.  We use the same trajectories as those simulated in the previous subsection.  Even though the goal is now to detect a bifurcation using high-frequency data, we still do not want to use all of the observations, or even all of the observations in a given window.  Instead, we sample the data and compute the RoD on the random samples.  

For this experiment, we use the same random samples of observations as in the previous subsection, we just utilize them differently.  Previously, we treated each sample of observations as if it were the only data we had to generate a binary prediction as to whether or not the system would undergo a Hopf bifurcation.  Now, we know that we are actually generating 100 predictions for each trajectory, and we use them to determine a probability that the system will undergo a bifurcation.  

Since we are now determining a `probability' for each trajectory, we can measure performance using area under the curve (AUC) where `the curve' refers to the receiver operating characteristic (ROC) curve.  The ROC curve depicts the performance of a binary classifier as the threshold used to separate positive predictions from negative is varied.  We are not required to use a probability of 0.5 to determine which trajectories we expect undergo bifurcation.  This observation means we no longer have to worry about false positives from individual samples, but rather on trajectories as a whole, allowing us to use RoD alone as a test statistic (e.g., without RMSSD or SD).  We found that RoD was best as a classifier when using a window length of 500 (half of the full trajectory before bifurcation) with $(\alpha,\beta) = (25,75)$, giving us an average of 10 observations in each RoD calculation.  The AUCs with these parameters were 0.937 for the `normal' system and 0.98 for the `excitable' system.  Figure \ref{fig:auc} depicts the ROC curves for the optimal case.  

\begin{figure}[t]
	\centering
	\begin{subfigure}[b]{0.45\textwidth}
		\includegraphics[width=\textwidth]{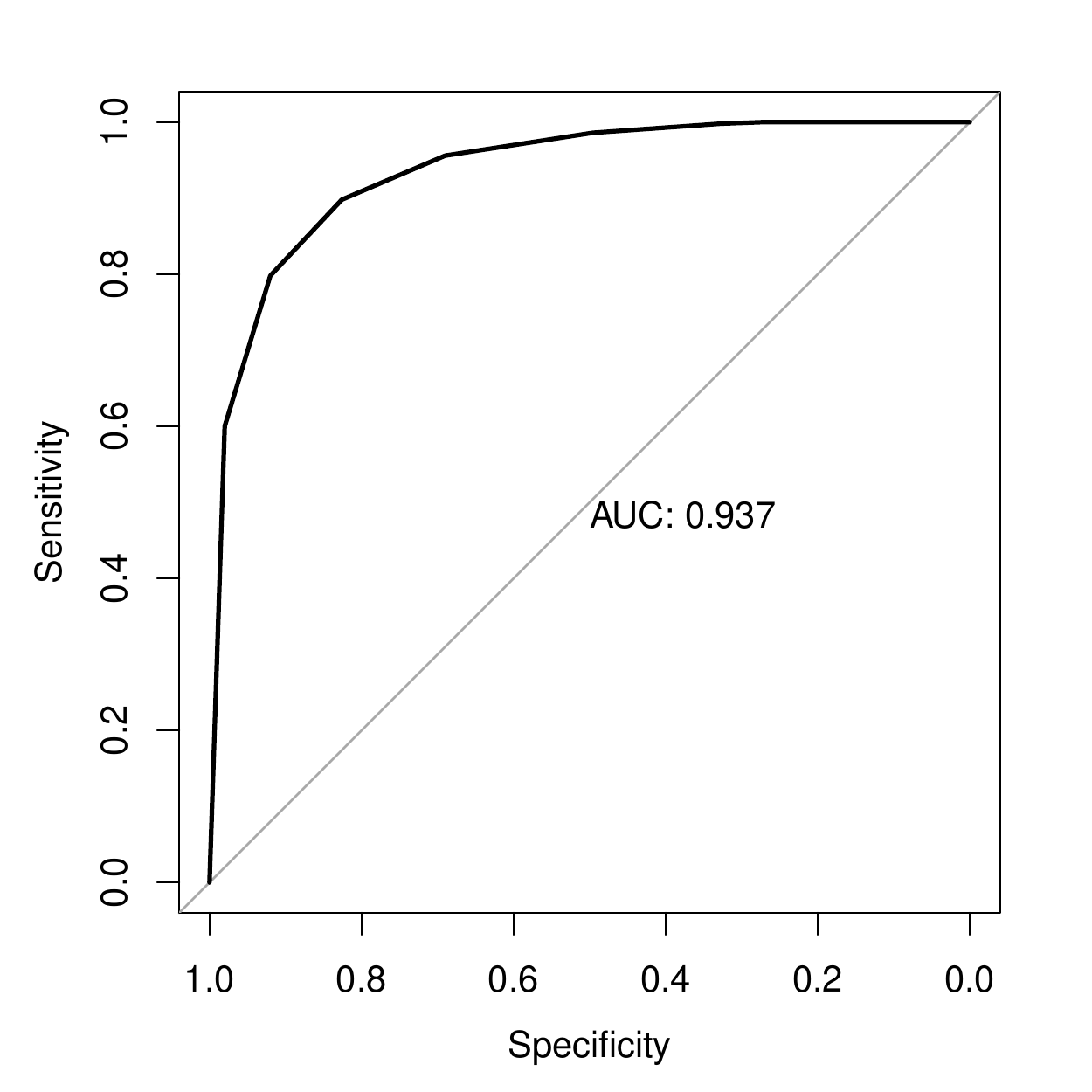}
		\caption{$a=10$ (Normal).}
		\label{fig:auc_std}
	\end{subfigure}
	~ 
	\begin{subfigure}[b]{0.45\textwidth}
		\includegraphics[width=\textwidth]{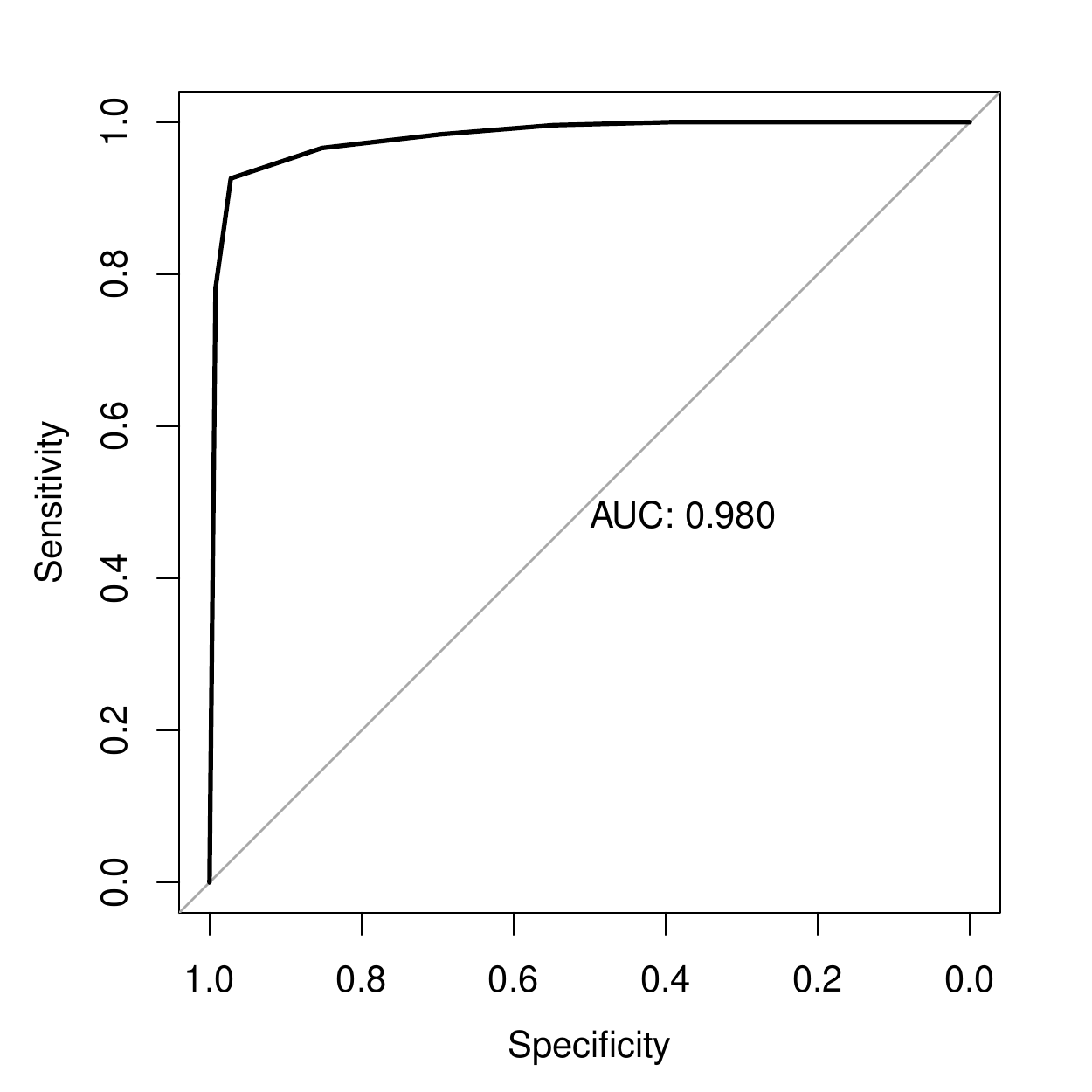}
		\caption{$a=1$ (Excitable).}
		\label{fig:auc_exp}
	\end{subfigure}
	\caption{ROC curves for two parametrizations of \eqref{vdp3} with $(\alpha,\beta) = (25,75)$ and window length = 500.}   
	\label{fig:auc} 
\end{figure}

In general, RoD performed well as a classifier in each of the experiments with between 8 and 14 observations expected to lie in each window.  We already discussed that a window length of 500 with an average of 10 observations per window gave the best results, but the second-best performance also came with an average of 10 observations per window (and window length of 750).  Table \ref{tbl:auc} lists the AUC for each of the experiments where the expected number of observations per window was in the interval $[8,14]$.  

\begin{table}[b] 
	\begin{center}
		\addtolength{\tabcolsep}{1mm}
		\renewcommand{\arraystretch}{1.2}
		\begin{tabular}{|c|c||c|c|c||c|c|}
			\hline
			\textbf{a}  & \textbf{AUC}	& \textbf{$\alpha$}	& \textbf{$\beta$}	& \textbf{Window}	& \textbf{a}	& \textbf{AUC}	\\
			\hline
			\hline  
			1 	 		& 0.867			& 25  				& 50				& 500				& 10 	 		& 0.830			\\
			\hline
			1 	 		& 0.980			& 25                & 75				& 500				& 10	 		& 0.937			\\
			\hline
			1  			& 0.876			& 50  				& 100				& 750				& 10			& 0.911			\\
			\hline
			1  			& 0.840			& 50  				& 100				& 1000				& 10 			& 0.800			\\
			\hline
		\end{tabular}
		\caption{Values used for experiments generating observations at random times.} 
		\label{tbl:auc}
	\end{center}
\end{table}

\section{Discussion}
\label{sec:discussion}

In this paper we presented a new early warning test statistic, RoD, that can be used with low-frequency data and non-uniform sampling rates.  RoD adequately detects Hopf bifurcations in situations where there is one relatively short time series with sparse observations.  It is interesting to note the role noise plays in the performance.  No false positives are detected in the absence of noise.  In the presence of noise, however, the magnitude of the noise does not affect the false positive rate.  On the other hand, stronger noise improves the true positive rate \emph{before} the bifurcation occurs.  Recalling Figure \ref{fig:del_hopf}, we emphasize that in the absence of noise (or when noise has sufficiently small influence), there will be a significant delay between when the bifurcation occurs and when it is observed.  As such, we should think of the true positive rates in Figures \ref{fig:exp_tp}-\ref{fig:std_fp} as lower bounds on the effectiveness of RoD as a detection metric.  In some applications, we may even wish to detect Hopf bifurcations that have already occurred.

In situations where one would like to detect an upcoming Hopf bifurcation from a longer, high-frequency time series, we recommend producing a number of short, sparse time series by sampling the high-frequency observations at irregular intervals.  When aggregating the results of RoD on the set of shorter time series, the performance is much stronger than on a single short time series.  

Many of the difficulties in tipping point prediction for climatological and ecological applications arise from having only one time series.  It would be interesting to see how the sampling technique affects performance of other early warning signals designed to detect catastrophic bifurcations.

\bibliographystyle{plain}
\bibliography{pews_refs}

\end{document}